\documentclass[10pt]{article}
\usepackage{amsmath}
\usepackage{amssymb}
\usepackage{amsthm}
\usepackage{lineno}
\usepackage[mathscr]{eucal}
\usepackage[english]{babel}
\usepackage{graphicx}
\theoremstyle{definition}
\newtheorem{definition}{Definition}[section]
\newtheorem{theorem}[definition]{Theorem}
\newtheorem*{theorem*}{Conjecture}
\newtheorem{proposition}[definition]{Proposition}
\newtheorem{lemma}[definition]{Lemma}

\theoremstyle{remark}
\newtheorem{remark}[definition]{Remark}

\newcounter{enumctr}

\newcommand{\N}{\mathbb{N}}
\newcommand{\R}{\mathbb{R}}

\newcommand{\C}{\mathbb{C}}
\newcommand{\id}{\hbox{id}}

\newcommand{\rT}{\mathrm {T}}
\providecommand{\keywords}[1]{\textbf{\textbf{Key words: }} #1}
\begin{document}


\title{\vspace*{-10mm}
A linearized stability theorem for nonlinear delay fractional differential equations}

\author{Hieu Trinh\footnote{\tt hieu.trinh@deakin.edu.au, \rm School of Engineering, Deakin University, Geelong, VIC 3217, Australia}
\;and\;
H.T.~Tuan\footnote{\tt httuan@math.ac.vn, \rm Institute of Mathematics, Vietnam Academy of Science and Technology, 18 Hoang Quoc Viet, 10307 Ha Noi, Viet Nam}}
\maketitle
\begin{abstract}
In this paper, we prove a theorem of linearized asymptotic stability for fractional differential
equations with a time delay. More precisely, using the method of linearization of a nonlinear equation along an orbit (Lyapunov's first method), we show that an equilibrium of a nonlinear Caputo
fractional differential equation with a time delay is asymptotically stable if its linearization at the equilibrium is asymptotically stable. Our approach based on a technique which converts the linear part of the equation into a diagonal one. Then using properties of generalized Mittag-Leffler functions, the construction of an associated Lyapunov--Perron operator and the Banach contraction mapping theorem, we obtain the desired result.
\end{abstract}

\section{Introduction}
Recently, delay fractional differential equations (DFDEs) have received considerable attentions because they provide mathematical models of real-world  problems in which the fractional rate of change depends on the influence of their hereditary effects, see e.g., \cite{Lakshmikantham,Benchohra,Krol,Cermak,Hieu} and the references therein.
One of the simplest form of DFDEs is 
\begin{equation}\label{add_eq}
\begin{cases}
^{C}D^\alpha_{0+}x(t)=f(t,x(t),x(t-\tau)),\quad t\in [0,T],\\
x(t)=\phi(t),\quad\forall t\in[-\tau,0],
\end{cases}
\end{equation}
where $\alpha>0$ is the order of the Caputo fractional derivative $^{C}D^\alpha_{0+}$, the initial condition $\phi$ is a continuous function on the interval $[-\tau,0]$ with $\tau,T>0$ are fixed real parameters. For this equation, the first basic and important problem is to show the existence and uniqueness of solutions under some reasonable conditions. It is well known that in the case of ordinary differential equations ($\alpha$ is an integer), under some Lipschitz conditions, a delay equation has an unique local solution (see \cite[Section 2.2]{Hale}); furthermore, by using continuation property (see \cite[Section 2.3]{Hale}), one can derive global solutions as well. However, in the fractional case (non-integer $\alpha$), the problem of existence and uniqueness of (local and global) solutions is more complex because of the  \textit{fractional order} feature of the equation which implies history dependence of the solutions, hence, among others, the continuation property is not applicable. With regard to the existence of solutions to DFDEs, many results have been reported in the literature, see e.g., Abbas \cite{Abbas} and N.D. Cong and H.T. Tuan \cite{Tuan}.

Furthermore, whenever the solution exists, it is of particular important to know the asymptotic behavior of them. To the best of our knowledge, up to now, there have been only very few contributions to this problem. Y. Luo and Y. Chen \cite{Luo}, K.A. Moornani and M. Haeri \cite{Moornani} discussed on the stability of some particular types of fractional differential equations with constant delays. J. Cermak, J. Hornicek and T. Kisela  \cite{Cermak} have discussed stability and asymptotitc properties of linear fractional-order differential systems involving both delayed as well as non-delayed terms. The stability and bifurcation analysis of a generalized scalar DFDE is discussed in \cite{Bhalekar16}. The stability and performance analysis for postive fractional-order systems with time-varying delays is reported in J. Shen and J. Lam \cite{Shen}. However, the relationship between the stability of the trivial solution to a nonlinear delay fractional differential system and that of the linearized part is still an open problem.

This paper is devoted to the investigation of the asymptotic behavior for solutions near the the equilibrium of \eqref{add_eq} in the case the function $f:\R^d\times \R^d\rightarrow \R^d$ has the form
\[
f(x,y)=Ay+g(x,y).
\]
Here, $A\in \R^{d\times d}$ and the function $g$ satisfies the following conditions
\begin{itemize}
\item[(H1)] $g(0,0)=0$;
\item[(H2)] $g$ is local Lipschitz continuous in a neighborhood of the origin and $$\lim_{\varrho\to 0}\ell_g(\varrho)=0,$$
with
$$\ell_g(\varrho):=\sup_{\substack{x,y,\hat{x},\hat{y}\in B_{\R^d}(0,\varrho)\\ (x,y)\neq (\hat{x},\hat{y})}}\frac{\|g(x,y)-g(\hat{x},\hat{y})\|}{\|x-\hat{x}\|+\|y-\hat{y}\|}.$$
\end{itemize}
Namely, we  prove that the trivial solution of \eqref{add_eq} is asymptotically stable if the trivial solution of the linearized equation 
\[
\begin{cases}
^{C}D^\alpha_{0+}x(t)=Ax(t-\tau),\quad t\in (0,\infty),\\
x(t)=\phi(t),\quad t\in [-\tau,0],
\end{cases}
\] 
where $\phi:[-\tau,0]\rightarrow \R^d$ is a continuous function, is asymptotically stable.

The rest of this paper is organized as follows. In Section 2, we recall briefly a framework of delay fractional differential systems. Section 3 is devoted to the main result of this paper. In this section, we give a spectrum characterization of the asymptotic stability to nonlinear fractional differential systems.
\section{Preliminaries}\label{sec.preliminary}
This section is devoted to recalling briefly a framework of DFDEs. We first introduce some notations which are used throughout this paper. Let $\Bbb K$ be the set of all real numbers or complex numbers and $\Bbb K^d$ be the $d$-dimensional Euclidean space endowed with a norm $\|\cdot\|$. Denote by $I$ the real interval $[a,b]$ or $[a,\infty)$, let $C(I;\Bbb K^d)$ be the space of continuous functions $\xi:I\rightarrow\Bbb K^d$ with the sup norm $\|\cdot\|_\infty$, i.e.,
\[
\|\xi\|_\infty:=\sup_{t\in I}\|\xi(t)\|,\quad\forall \xi\in C(I;\Bbb K^d).
\]
Final, we denote by $B_{C([a,b];\Bbb K^d)}(0,\varrho)$  the ball centered at the origin with radius $\varrho$ in the space $C([a,b];\Bbb K^d)$ and $B_{C_\infty}(0,\varrho)$  the ball with the center at the origin and radius $\varrho$ in the space $C([a,\infty);\Bbb K^d)$.

For $\alpha>0$, $[a,b]\subset \R$ and a measurable function $x:[a,b]\rightarrow \R$  such that $\int_a^b|x(\tau)|\;d\tau<\infty$, the Riemann--Liouville integral operator of order $\alpha$ is defined by
\[
(I_{a+}^{\alpha}x)(t):=\frac{1}{\Gamma(\alpha)}\int_a^t(t-s)^{\alpha-1}x(s)\;ds,\quad t\in (a,b],
\]
where $\Gamma$ is the Gamma function. The \emph{Caputo fractional derivative} $^{C\!}D_{a+}^\alpha x$ of a function $x\in AC^m([a,b];\R)$ is defined by
\[
(^{C\!}D_{a+}^\alpha x)(t):=(I_{a+}^{m-\alpha}D^mx)(t),\quad t\in (a,b],
\]
where $AC^m([a,b];\R)$ denotes the space of real functions $x$ which has continuous derivatives up to order $m-1$ on the interval $[a,b]$ and the $({m-1})^{th}$-order derivative $x^{(m-1)}$ is absolutely continuous, $D^m=\frac{d^m}{dt^m}$ is the usual $m^{th}$-order derivative and $m:=\lceil\alpha\rceil$ is the smallest integer larger or equal to $\alpha$. The Caputo fractional derivative of a $d$-dimensional vector function $x(t)=(x_1(t),\cdots,x_d(t))^{\rT}$ is defined component-wise as
\[
(^{C\!}D_{a+}^\alpha x)(t):=(^{C\!}D_{a+}^\alpha x_1(t),\cdots,^{C\!\!}D_{a+}^\alpha x_d(t))^{\rT}.
\]
From now on, we consider only the case $\alpha\in (0,1)$. Let $\tau$ be an arbitrary positive constant, and  $\phi\in C([-r,0];\R^d)$ be a given
continuous function. Consider the delay Caputo fractional differential equations
\begin{equation}\label{IntroEq}
^{C\!}D_{0+}^\alpha x(t)=Ax(t-\tau)+g(x(t),x(t-\tau)),\quad t\in [0,\infty),
\end{equation}
with the initial condition 
\begin{equation}\label{Ini_Cond}
x(t)=\phi(t),\quad \forall t\in[-\tau,0],
\end{equation} 
where $x\in \R^d$, $A\in \R^{d\times d}$ and $g:\R^d\times\R^d \rightarrow \R^d$ is local Lipschitz continuous in a neighborhood of the origin.

For any $T>0$, a function $\varphi(\cdot,\phi)\in C([-\tau,T];\R^d)$ is called a \emph{solution} of the initial condition problem \eqref{IntroEq}--\eqref{Ini_Cond} over the interval $[-\tau,T]$ if 
\begin{equation*}
\begin{cases}
^{C\!}D^\alpha_{0+}\varphi(t,\phi)=Ax(t-\tau)+g(\varphi(t,\phi),\varphi(t-\tau,\phi)),\quad \forall t\in (0,T],\\
\varphi(t,\phi)=\phi(t),\quad \forall t\in[-\tau,0].
\end{cases}
\end{equation*}
Since $g$ is local Lipschitz continuous in a neighborhood of the origin, \cite[Theorem 3.1]{Tuan} implies the  existence and uniqueness of solutions to the initial value problem \eqref{IntroEq}--\eqref{Ini_Cond} for any $\phi\in C([-\tau,0];\R^d)$. Let $I:=[-\tau,t_{\max}(\phi))$, where $0<t_{\max}(\phi)\leq \infty$, be the maximal interval of existence to the solution $\varphi(\cdot,\phi)$. We now recall the nations of stability and asymptotic stability of the trivial solution to the equation \eqref{IntroEq}.
\begin{definition}
\begin{itemize}
\item [(i)] The trivial solution of \eqref{IntroEq} is called stable if for any $\varepsilon>0$ there exists $\delta=\delta(\varepsilon)>0$ such that for any $\|\phi\|_\infty\leq \delta$, we have $t_{\max}(\phi)=\infty$ and
\[
\|\varphi(t,\phi)\|\leq \varepsilon,\quad \forall t\geq 0.
\]
\item[(ii)] The trivial solution is called asymptotic stable if it is stable and there exists $\hat{\delta}>0$ such that $\lim_{t\to\infty}\varphi(t,\phi)=0$ whenever $\|\phi\|_\infty\leq \hat{\delta}$.
\end{itemize}
\end{definition}
In the case $g=0$, the equation \eqref{IntroEq} reduces to a linear delay fractional equation 
\begin{equation}\label{linearEq}
^{C}D^\alpha_{0+} x(t)=Ax(t-\tau),\quad \forall t\geq 0.
\end{equation}
Let
\[
\mathcal{S}_{\alpha,\tau}:=\left\{\lambda\in\C\setminus\{0\}:|\lambda|<\left(\frac{|\arg{(\lambda)}|-\alpha\pi/2}{\tau}\right)^{\alpha}, \frac{\alpha \pi}{2}<|\arg{(\lambda)}|\leq \pi\right\}.
\]
In the following theorem, we restate a spectral characterization on the asymptotic stability of the trivial solution to \eqref{linearEq}.
\begin{theorem}
The trivial solution of \eqref{linearEq} is asymptotic stable if and only if all eigenvalues of the matrix $A$ are located in the domain $\mathcal{S}_{\alpha,\tau}$, i.e.,
\[
\sigma(A)\subset \mathcal{S}_{\alpha,\tau}.
\]
\end{theorem}
\begin{proof}
See \cite[Theorem 2]{Cermak}.
\end{proof}
For the nonlinear equation \eqref{IntroEq}, we first focus on the case: the matrix $A$ is diagonal and the function $g$ is global Lispchitz continuous. By using the generalized Mittag-Leffler function $E^{\lambda,\tau}_{\alpha,\beta}(t):[0,\infty)\rightarrow \C$, which is defined by
\[
E^{\lambda,\tau}_{\alpha,\beta}(t):=\begin{cases}
\sum_{k=0}^{\infty}\frac{\lambda^k (t-k\tau)^{\alpha k+\beta-1}}{\Gamma(\alpha k+\beta)}H(t-k\tau),\quad\text{if}\;\; t\geq 0,\\
1,\quad \text{if}\;\; t=0,
\end{cases}
\]
where $\beta\in R$, $\lambda\in \C$ and $H$ is the Heaviside function defined by
\[
H(t)=\begin{cases}
1,\quad \text{if}\;\; t\geq 0,\\
0,\quad \text{if}\;\; t<0,
\end{cases}
\]
we obtain a connection between the solutions of the equation \eqref{IntroEq} and its linear part as below.
\begin{lemma}\label{Equivalent_eq}
Consider the initial problem \eqref{IntroEq}--\eqref{Ini_Cond}. Assume that $g$ is global Lipschitz continuous and $$A=\hbox{diag}(\lambda_1,\dots,\lambda_d),$$
where $\lambda_i\in \C$, for $i=1,\dots,d$. Then, for any initial condition $\phi\in C([-\tau,0];\C^d)$, this problem has a unique solution on $[-\tau,\infty)$. Denote this solution by $\varphi(\cdot,\phi)$. We have a representation of $\varphi(\cdot,\phi)$ as $\varphi(\cdot,\phi):=(\varphi^1(\cdot,\phi),\dots,\varphi^d(\cdot,\phi))^{\rm T}$, in which, for $i\in \{1,\dots, d\}$,
\begin{align}
\notag \varphi^i(t,\phi):=& E^{\lambda_i,\tau}_{\alpha,1}(t)\phi^i(0)+\lambda_i \int_{-\tau}^0 E^{\lambda_i,\tau}_{\alpha,\alpha}(t-\tau-s)H(t-\tau-s)\phi^i(s)ds\\
\label{invariant_for}&+\int_0^t E^{\lambda_i,\tau}_{\alpha,\alpha}(t-s)g^i(\varphi(s,\phi),\varphi(s-\tau,\phi))ds,\quad \forall t\geq 0,
\end{align} 
and $\varphi(t,\phi)=\phi(t)$ for all $t\in [-\tau,0]$. 
\end{lemma}
\begin{proof}
From \cite[Corollary 3.2]{Tuan}, we see that the initial problem \eqref{IntroEq}--\eqref{Ini_Cond} has a unique solution with any initial condition $\phi\in C([-\tau,0];\C^d)$. On the other hand, due to \cite[Theorem 4.1]{Tuan}, all solutions of this problem are exponential bounded. Using Laplace transform and arguments as in \cite[Section 4]{Cermak}, we obtain the variation of constants formula \eqref{invariant_for}.
\end{proof}
In the remaining part of this section, we give some estimates involving the scalar generalized Mittag-Leffler functions $E^{\lambda,\tau}_{\alpha,\beta}$ with $\lambda\in\mathcal{S}_{\alpha,\tau}$ and $\beta=1$ or $\beta=\alpha$.
\begin{lemma}\label{estGML}
Assume that $\lambda \in \mathcal{S}_{\alpha,\tau}$. Then, there exits a positive constant $C_{\alpha,\lambda}$ such that the following statement hold:
\begin{itemize}
\item[(i)] $|E_{\alpha,\alpha}^{\lambda,\tau}(t)|\leq \frac{C_{\alpha,\lambda}}{t^{\alpha+1}}$,\quad $\forall t\geq 1$;
\item[(ii)] $|E_{\alpha,1}^{\lambda,\tau}(t)|\leq \frac{C_{\alpha,\lambda}}{t^\alpha}$,\quad $\forall t\geq 1$;
\item[(iii)] $
\sup_{t\geq 0}\int_0^t |E^{\lambda,\tau}_{\alpha,\alpha}(s)|\;ds\leq C_{\alpha,\lambda}.$
\end{itemize}
\end{lemma}
\begin{proof}
This proof is given in the appendix at the end of this paper.
\end{proof}
\section{Main result}
Our aim in this section is to prove the following theorem.
\begin{theorem}[Linearized stability theorem]\label{Main result}
Consider the initial problem \eqref{IntroEq}--\eqref{Ini_Cond}. Assume that the spectrum $\sigma(A)$ of the matrix $A$ satisfies
\[
\sigma(A)\subset \mathcal{S}_{\alpha,\tau}
\]
and the function $g$ satisfies the conditions $(\textup{H}_1)$ and $(\textup{H}_2)$. Then, the trivial solution of this problem is asymptotically stable.
\end{theorem}
For a proof of this theorem, we follow the approach of \cite{Cong1}. More precesily, first we transform the linear part of \eqref{IntroEq} to a Jordan normal form; then we construct an appropriate Lyapunov--Perron operator which is a contraction and its fiexed point is the solution of the initial problem \eqref{IntroEq}--\eqref{Ini_Cond}, and exploit the properties of the scalar generalized Mittag-Leffler function to obtain the conclusion of the theorem.

{\bf Transformation of the linear part} 

Using \cite[Theorem 6.37, pp.~146]{Shilov}, there exists a nonsingular matrix $T\in\C^{d\times d}$ transforming the matrix $A$ in the equation \eqref{IntroEq} into the Jordan normal form, i.e.,
\[
T^{-1}A T=\hbox{diag}(A_1,\dots,A_n),
\]
for $i=1,\dots,n$, the block $A_i$ is of the following form
\[
A_i=\lambda_i\, \id_{d_i\times d_i}+\eta_i\, N_{d_i\times d_i},
\]
where $\eta_i\in\{0,1\}$ and the nilpotent matrix $N_{d_i\times d_i}$ is given by
\[
N_{d_i\times d_i}:=
\left(
      \begin{array}{*7{c}}
      0  &     1         &    0      & \cdots        &  0        \\
        0        & 0    &    1     &   \cdots      &              0\\
        \vdots &\vdots        &  \ddots         &          \ddots &\vdots\\
        0 &    0           &\cdots           &  0 &          1 \\

        0& 0  &\cdots                                          &0         & 0 \\
      \end{array}
    \right)_{d_i \times d_i}.
\]
Let $\gamma$ be an arbitrary but fixed positive number. Using the transformation $P_i:=\textup{diag}(1,\gamma,\dots,\gamma^{d_i-1})$, we obtain that
\begin{equation*}
P_i^{-1} A_i P_i=\lambda_i\, \id_{d_i\times d_i}+\gamma_i\, N_{d_i\times d_i},
\end{equation*}
$\gamma_i\in \{0,\gamma\}$. Hence, under the transformation $y:=(TP)^{-1}x$, the equation \eqref{IntroEq} becomes
\begin{equation}\label{NewSystem}
^{C}D_{0+}^\alpha y(t)=\hbox{diag}(J_1,\dots,J_n)y(t-\tau)+h(y(t),y(t-\tau)),
\end{equation}
where $J_i:=\lambda_i \id_{d_i\times d_i}$ for $i=1,\dots,n$ and the function $h$ is given by
\begin{align}
\notag h(y(t),y(t-\tau)):=&\text{diag}(\gamma_1N_{d_1\times d_1},\dots,\gamma_nN_{d_n\times d_n})y(t-\tau)\\
&\hspace*{0.5cm}+(TP)^{-1}f(TPy(t),TPy(t-\tau)).\label{Eq3}
\end{align}
\begin{remark}\label{Remark1}
The function $h$ in the equation \eqref{NewSystem} is local Lipschitz continuous in a neighborhood of the origin and
\[
h(0,0)=0,\quad\text{and}\quad \lim_{\varrho\to 0}\ell_h(\varrho)=
\left\{
\begin{array}{ll}
\gamma & \hbox{if there exists } \gamma_i=\gamma,\\[1ex]
0 & \hbox{otherwise}.
\end{array}
\right.
\]
\end{remark}
\begin{remark}\label{Remark2}
If the trivial solution of equations \eqref{NewSystem} is stable (or asymptotically stable), then the trivial of \eqref{IntroEq} is the same, i.e., it is also stable (or asymptotically stable).
\end{remark}
{\bf Construction of an appropriate Lyapunov-Perron operator}

We are now introducing a Lyapunov-Perron operator associated with \eqref{NewSystem}. Before doing this, we discuss some conventions which are used in the remaining part of this section: The space $\C^d$ can be written as $\C^d=\C^{d_1}\times\dots\times\C^{d_n}$. A vector $x\in\C^d$ can be written component-wise as $x=(x^1,\dots,x^n)$.

For any $\phi\in C([-\tau,0];\C^d)$, the operator $\mathcal{T}_{\phi,r}: C([-r,\infty);\C^d)\rightarrow C([-\tau,\infty);\C^d)$ is defined by
\[
(\mathcal{T}_{\phi,\tau}\xi)(t)=((\mathcal{T}_{\phi,\tau}\xi)^1(t),\dots,(\mathcal{T}_{\phi,\tau}\xi)^n(t))^{\rm T},
\]
where for $i=1,\dots,n$ and $t\geq 0$
\begin{eqnarray*}
(\mathcal{T}_{\phi,\tau}\xi)^i(t)
&=&
E^{\lambda_i,\tau}_{\alpha,1}(t) \phi^i(0)+\lambda_i\int_{-\tau}^0 E^{\lambda_i,\tau}_{\alpha,\alpha}(t-\tau-s)H(t-\tau-s)\phi^i(s)\;ds \\
&&
+\int_0^t E^{\lambda_i,\tau}_{\alpha,\alpha}(t-s)h^i(\xi(s),\xi(s-\tau))\,ds,
\end{eqnarray*}
and 
\[
(\mathcal{T}_{\phi,\tau}\xi)(t)=\phi(t),\quad \forall t\in[-\tau,0],
\]
is called the \emph{Lyapunov-Perron operator associated with \eqref{NewSystem}}. 
Next, we provide some estimates on the operator $\mathcal{T}_{\phi,\tau}$.
\begin{proposition}\label{Prp2} Consider system \eqref{NewSystem} and suppose that
\[
\sigma{(A)}\subset \mathcal{S}_{\alpha,\tau}.
\]
Let $\varepsilon_1$ be a small positive parameter such that the function $h$ is Lipschitz continuous on $B_{C_\infty}(0,\varepsilon_1)\times B_{C_\infty}(0,\varepsilon_1)$. Then, for any $\xi,\hat\xi\in B_{C_\infty}(0,\varepsilon_1)$, we have
\begin{align*}
&\|\mathcal{T}_{\phi,\tau}\xi-\mathcal{T}_{\hat{\phi},\tau}\hat{\xi}\|_{\infty}\leq  \max\Big\{\max_{1\leq i\leq n}\sup_{t\geq 0}\Big\{|E^{\lambda_i,\tau}_{\alpha,1}(t)|+|\lambda_i|\int_{t-\tau}^t |E^{\lambda_i,\tau}_{\alpha,\alpha}(s)|\;ds\Big\}\times \|\phi-\hat\phi\|_{\infty}\\
&\hspace{1cm}+\max_{1\leq i\leq n}\int_0^\infty |E^{\lambda_i,\tau}_{\alpha,\alpha}(s)|\;ds\times \ell_h(\hat \varepsilon_1)\times\|\xi-\hat{\xi}\|_\infty,\|\phi-\hat{\phi}\|_\infty\Big\},
\end{align*}
for all $\phi,\hat{\phi}\in B_{C([-\tau,0];\C^d)}(0,\varepsilon_1)$.
\end{proposition}
\begin{proof}
For $i=1,\dots,n$ and $t\geq 0$, we get
\begin{align*}
\|(\mathcal{T}_{\phi,\tau} \xi)^i(t)-(\mathcal{T}_{\hat{\phi},\tau} \hat{\xi})^i(t)\| \le &
\;\|\phi-\hat \phi\|_\infty \Big(|E^{\lambda_i,\tau}_{\alpha,1} (t)| +|\lambda_i|\int_{t-\tau}^{t} |E^{\lambda_i,\tau}_{\alpha,\alpha} (s)|\;ds\Big)\\
&\hspace{-30mm}+\ell_h(\max\{\|\xi\|_\infty,\|\widehat{\xi}\|_\infty\})\times \|\xi-\widehat{\xi}\|_\infty \int_0^t|E^{\lambda_i,\tau}_{\alpha,\alpha} (s)|\;ds.
\end{align*}
Hence, for any $\xi,\hat\xi\in B_{C_\infty}(0,\varepsilon_1)$, we have
\begin{align*}
&\|\mathcal{T}_{\phi,\tau}\xi-\mathcal{T}_{\hat{\phi},\tau}\hat{\xi}\|_{\infty}\leq  \max\Big\{\max_{1\leq i\leq n}\sup_{t\geq 0}\Big\{|E^{\lambda_i,\tau}_{\alpha,1}(t)|+|\lambda_i|\int_{t-\tau}^t |E^{\lambda_i,\tau}_{\alpha,\alpha}(s)|\;ds\Big\}\times \|\phi-\hat\phi\|_{\infty}\\
&\hspace{1cm}+\max_{1\leq i\leq n}\int_0^\infty |E^{\lambda_i,\tau}_{\alpha,\alpha}(s)|\;ds\times \ell_h(\varepsilon_1)\times\|\xi-\hat{\xi}\|_\infty,\|\phi-\hat{\phi}\|_\infty\Big\},
\end{align*}
for all $\phi,\hat{\phi}\in B_{C([-\tau,0];\C^d)}(0,\varepsilon_1)$.
The proof is complete.
\end{proof}
From the proposition above, by letting $C(\lambda,\alpha):=\max_{1\le i\le n}\int_0^\infty |E_{\alpha,\alpha}^{\lambda_i,\tau}(s)|\;ds$, for any $\xi,\hat\xi\in B_{C_\infty}(0,\varepsilon_1)$, we have
\begin{equation*}
\|\mathcal{T}_{\phi,\tau}\xi-\mathcal{T}_{{\phi},\tau}\hat{\xi}\|_{\infty}\leq  C(\lambda,\alpha)\times \ell_h(\varepsilon_1)\times\|\xi-\hat{\xi}\|_\infty,
\end{equation*}
for all $\phi\in C([-\tau,0];\C^d)$.
Note that the Lipschitz constant
$C(\alpha,\lambda)$ is independent of the constant $\varepsilon_1$ which is hidden in the coefficients of system \eqref{NewSystem}. From now on, we
choose and fix the constant $\varepsilon_2$ as $\varepsilon_2:=\min\{\varepsilon_1,\frac{1}{2C(\alpha,\lambda)}\}$. The remaining question is now to choose a ball
with small radius in $C_\infty(\R_{\geq 0},\C^d)$ such that the restriction of the Lyapunov-Perron operator to this ball is
strictly contractive.
\begin{lemma}\label{lemma6}
The following statements hold:
\vspace{-5mm}
\begin{itemize}
\item [(i)] There exists $\varepsilon>0$ such that
\begin{equation}\label{Eq7a}
q:=C(\alpha,\lambda)\times  \ell_h(\varepsilon) < 1.
\end{equation}
\item [(ii)] Choose and fix $\varepsilon>0$ satisfying \eqref{Eq7a}. Define
\begin{equation}\label{Eq7b}
\delta=\frac{\varepsilon(1-q)}{\max_{1\le i\le n}\sup_{t\geq 0}\Big\{|E_{\alpha,1}^{\lambda_i,\tau}(t)|+|\lambda_i|\int_{t-\tau}^t |E^{\lambda_i,\tau}_{\alpha,\alpha}(s)|ds+1\Big\}}.
\end{equation}
Then, for any $\phi\in
B_{C([-\tau,0];\C^d)}(0,\delta)$, we have $\mathcal T_{\phi,\tau} (B_{C_{\infty}}(0,\varepsilon))\subset B_{C_{\infty}}(0,\varepsilon)$ and
\begin{equation*}\label{LipschitzContinuity}
\|\mathcal T_{\phi,\tau}\xi-\mathcal T_{\phi,\tau}\hat {\xi}\|_\infty \leq q\|\xi-\hat{\xi}\|_\infty\quad\hbox{ for all } \xi,\hat{\xi}\in
B_{C_{\infty}}(0,\varepsilon).
\end{equation*}
\end{itemize}
\end{lemma}
\begin{proof}
\noindent (i) By Remark \ref{Remark1}, $\lim_{\varrho\to 0}\ell_h(\varrho)\leq \gamma$. Hence, we can choose a positive constant $\varepsilon$ such that 
\[
q:=C(\alpha,\lambda)\times  \ell_h(\varepsilon) < 1,
\]
and the assertion (i) is proved.

\noindent (ii) According to
Proposition \ref{Prp2}, for any $\phi\in B_{C([-\tau,0];\C^d)}(0, \delta)$ and any $\xi \in B_{C_\infty}(0,\varepsilon)$, we obtain that
\begin{align*}
\|\mathcal T_{\phi, \tau}\xi\|_\infty &\leq \;\max_{1\le i \le n}\sup_{t\geq 0}\Big\{|E_{\alpha,1}^{\lambda_i,\tau}(t)|+|\lambda_i|\int_{t-\tau}^t |E_{\alpha,\alpha}^{\lambda_i,\tau}(s)|\;ds+1\Big\}\times \|\phi\|_\infty\\
&\hspace*{1cm}+ C(\alpha,\lambda)\times \ell_{h}(\varepsilon)\times\|\xi\|_{\infty}\\
& \leq \;(1-q)\varepsilon+q\varepsilon,
\end{align*}
which proves that $\mathcal T_{\phi,\tau}(B_{C_{\infty}}(0,\varepsilon))\subset B_{C_{\infty}}(0,\varepsilon)$. Furthermore, we also have
\begin{eqnarray*}
\|\mathcal T_{\phi,\tau}\xi-\mathcal T_{\phi,\tau}\widehat{\xi}\|_\infty &\leq&
C(\alpha,\lambda)\times \ell_{h}(\varepsilon)\times\;\|\xi-\widehat{\xi}\|_\infty\\[1.5ex]
&\leq & q\|\xi-\widehat{\xi}\|_\infty,
\end{eqnarray*}
which concludes the proof.
\end{proof}
\begin{proof}[Proof of Theorem \ref{Main result}]
Due to Remark \ref{Remark2}, it is sufficient to prove the asymptotic stability for the trivial solution of system \eqref{NewSystem}. For this
purpose, let $\delta$ be defined as in \eqref{Eq7b}. Let $\phi\in B_{C([-\tau,0],\C^d}(0,\delta)$ be arbitrary. Using Lemma \ref{lemma6} and the Contraction Mapping
Principle, there exists a unique fixed point $\xi \in B_{C_\infty}(0,\varepsilon)$ of $\mathcal{T}_{\phi,\tau}$. According to Lemma \ref{Equivalent_eq}, this point is also a solution of
\eqref{NewSystem} with the initial condition $\xi(t)=\phi(t)$ for all $t\in [-\tau,0]$. Since the equation
\eqref{NewSystem} has unique global solution in $B_{C_\infty}(0,\varepsilon)$ for each initial condition $\phi\in B_{C([-\tau,0],\C^d}(0,\delta)$, the trivial solution is stable. To complete the proof of the theorem, we have to show that the trivial solution is
attractive. Suppose that $\xi(t)=(\xi^1(t),\dots,\xi^n(t))$ is the solution of \eqref{NewSystem} which satisfies $\xi(t)=\phi(t)$ for every $t\in[-\tau,0]$, where $\phi\in B_{C([-\tau,0];\C^d)}(0,\delta)$. From Lemma \ref{lemma6}, we see that $\|\xi\|_\infty\le \varepsilon$. Put
$a:=\limsup_{t\to\infty}\|\xi(t)\|$, then $a\in [0,\varepsilon]$. Let $\hat\varepsilon$ be a positive number small enough. Then, there exists $T(\hat\varepsilon)>0$
such that
\[
\|\xi(t)\|\le (a+\hat\varepsilon)\qquad \textup{for any}\;t\ge T(\hat\varepsilon).
\]
For each $i=1,\dots,n$, we will estimate $\limsup_{t\to\infty}\|\xi^i(t)\|$. According to Lemma \ref{estGML}(i) and \ref{estGML}(ii), we obtain
\begin{itemize}
\item[(i)] $\lim_{t\to \infty}E_{\alpha,1}^{\lambda_i,\tau}(t)=0$;
\item[(ii)] $\lim_{t\to\infty}\int_{-\tau}^0E_{\alpha,\alpha}^{\lambda_i,\tau}(t-\tau-s)H(t-\tau-s)\phi^i(s)\;ds=0$;
\item[(iii)]
\begin{eqnarray*}
&&\limsup_{t\to\infty}\left\|\int_0^{T(\hat\varepsilon)}E_{\alpha,\alpha}^{\lambda_i,\tau}(t-s)h^i(\xi(s))\;ds\right\|\\[1.5ex]
&\le& \max_{t\in [0,T(\varepsilon)]}\|h^i(\xi(t))\|\limsup_{t\to \infty}\int_0^{T(\hat\varepsilon)}\frac{C_{\alpha,\lambda_i}}{(t-s)^{\alpha+1}}ds\\
&= &0.
\end{eqnarray*}
\end{itemize}
Therefore, from the fact that $\xi^i(t)=(\mathcal{T}_x \xi)^i(t)$, we have
\begin{eqnarray*}
\limsup_{t\to\infty}\|\xi^i(t)\| &=&
\limsup_{t\to\infty}\left\|\int_{T(\hat\varepsilon)}^tE_{\alpha,\alpha}^{\lambda_i,\tau}(t-s)h^i(\xi(s))ds\right\|\\
&\le& \ell_h(\varepsilon)\times C_{\alpha,\lambda_i}\times (a+\hat\varepsilon),
\end{eqnarray*}
where we use the estimate
\begin{eqnarray*}
\left|\int_{T(\hat\varepsilon)}^tE_{\alpha,\alpha}^{\lambda_i,\tau}(t-s)\;ds\right| &=&
\left|\int_0^{t-T(\hat\varepsilon)}E_{\alpha,\alpha}^{\lambda_i,\tau}(u)\;du\right|\\
&\le& C_{\alpha,\lambda_i},
\end{eqnarray*}
see Lemma \ref{estGML}(iii), to obtain the inequality above. Thus,
\begin{align*}
a &\le \max\left\{\limsup_{t\to\infty}\|\xi^1(t)\|,\dots,\limsup_{t\to\infty}\|\xi^n(t)\|\right\}\\
& \le  \ell_h(\tau)\times C(\alpha,\lambda)\times (a+\hat\varepsilon).
\end{align*}
Letting $\hat\varepsilon\to 0$, we have
\[
a \le \ell_h(\varepsilon)\times C(\alpha,\lambda)\times a.
\]
Due to the fact $\ell_h(\varepsilon)\times C(\alpha,\lambda)<1$, we get that $a=0$ and the proof is complete.
\end{proof}
\section{Conclusions}
This paper has studied the asymptotic behavior of solutions to nonlinear fractional differential equations with a time delay. We have shown that an equilibrium of a nonlinear Caputo
fractional differential equation with a time delay is asymptotically stable if its linearization at the equilibrium is asymptotically stable, that is, we have give a sufficient condition of the asymptotic stability basing on the characteristic spectrum of the linear part to the original equation. This is a new contribution in the qualitative theory of nonlinear fractional differential equations with delays. In the future, we hope to obtain a characteristic spectrum for the stability of fractional differential equations with multi-delays in high dimensional spaces. 
\section*{Acknowledgement}
The second author is supported by the Vietnam National Foundation for
Science and Technology Development (NAFOSTED) under Grant No. 101.03--2017.01.
\section*{Appendix}
\begin{proof}[Proof of Lemma \ref{estGML}]

For $\mu>0$ and $\theta\in (0,\pi)$, we denote $\gamma(\mu,\theta)$ the oriented contour formed by three segments:
\begin{itemize}
\item $\arg{(z)}=-\theta$,\quad $|z|\geq \mu$;
\item $-\theta\leq \arg{(z)}\leq \theta$,\quad $|z|=\mu$;
\item $\arg{(z)}=\theta$,\quad $|z|\geq \mu$.
\end{itemize}
From \cite[Proposition 1 (ii)]{Cermak}, we can choose a positive constant $\delta$ such that all zeros $z_i$ of the function $s^\alpha-\lambda \exp{(-\tau s)}$ satisfy $|\arg{(z_i)}|\ne \frac{\pi}{2}+\delta$, and there are only finitely many of them satisfying $|\arg{(z_i)}|\leq \frac{\pi}{2}+\delta$. Hence, there exist $R>\varepsilon>0$ such that all $z_i$ lie to the left of $\gamma(R,\frac{\pi}{2}+\delta)$ and those satisfying $|\arg{(z_i)}|\leq \frac{\pi}{2}+\delta$ are located to the right of $\gamma(\varepsilon,\frac{\pi}{2}+\delta)$, see \cite[p. 116]{Cermak}. Let $\beta\in \{1,\alpha\}$. For $t\geq 1$, from \cite[p. 116]{Cermak}, we have
\begin{align*}
E_{\alpha,\beta}^{\lambda,\tau}(t)&=\frac{1}{2\pi i}\int_{\gamma(R,\frac{\pi}{2}+\delta)}\frac{s^{\alpha-\beta}\exp{(ts)}}{s^\alpha-\lambda \exp{(-\tau s)}}\;ds\\
&=I^1(t)+I^2(t),
\end{align*}
where
\[
I^1(t)=\frac{1}{2\pi i}\int_{\gamma(\frac{\varepsilon}{t},\frac{\pi}{2}+\delta)}\frac{s^{\alpha-\beta}\exp{(ts)}}{s^\alpha-\lambda \exp{(-\tau s)}}\;ds,
\]
and
\[
I^2(t)=\frac{1}{2\pi i}\int_{\substack{\gamma(R,\frac{\pi}{2}+\delta)-\gamma(\frac{\varepsilon}{t},\frac{\pi}{2}+\delta)}}\frac{s^{\alpha-\beta}\exp{(ts)}}{s^\alpha-\lambda \exp{(-\tau s)}}\;ds.
\]
For $I^1(t)$, we use the representation
\[
I^1(t)=I^1_1(t)+I^1_2(t)+I^1_3(t),
\]
where
\[
I^1_1(t)=-\frac{1}{\lambda 2 \pi i}\int_{\gamma(\frac{\varepsilon}{t},\frac{\pi}{2}+\delta)}s^{\alpha-\beta}\exp{((\tau+t)s)}\;ds,
\]
\[
I^1_2(t)=-\frac{1}{\lambda^2 2 \pi i}\int_{\gamma(\frac{\varepsilon}{t},\frac{\pi}{2}+\delta)}s^{2\alpha-\beta}\exp{((2\tau+t)s)}\;ds,
\]
\[
I^1_3(t)=\frac{1}{\lambda^ 2 2\pi i}\int_{\gamma(\frac{\varepsilon}{t},\frac{\pi}{2}+\delta)}\frac{s^{3 \alpha-\beta}\exp{((2\tau+t)s)}}{s^\alpha-\lambda \exp{(-\tau s)}}\;ds.
\]
Using the change of variable $s=\frac{u^{1/\alpha}}{t}$, we have
\[
I^1_1(t)=-\frac{1}{\lambda}\frac{1}{ 2\alpha \pi i}\int_{\gamma(\varepsilon^\alpha,\frac{\alpha\pi}{2}+\delta)} u^{\frac{1-\beta}{\alpha}}\exp{((1+\frac{\tau}{t})u^{1/\alpha})}\;du\times \frac{1}{t},
\]
which, by changing the variable $\nu=(1+\frac{\tau}{t})^\alpha u$, impplies
\begin{align}\label{e1}
\notag I^1_1(t)=&-\frac{1}{\lambda}\frac{1}{2\alpha\pi i}\int_{\gamma((1+\frac{\tau}{t})^\alpha\varepsilon^\alpha,\frac{\alpha\pi}{2}+\delta)}\frac{\nu^{\frac{1-\beta}{\alpha}}\exp{(\nu^{1/\alpha})}}{(1+\frac{\tau}{t})^{\alpha-\beta+1}}\;d\nu\times \frac{1}{t^{\alpha-\beta+1}}\\
\notag =&-\frac{1}{\lambda}\frac{1}{2\alpha\pi i}\int_{\gamma((1+\frac{\tau}{t})^\alpha\varepsilon^\alpha,\frac{\alpha\pi}{2}+\delta)}\nu^{\frac{1-\beta}{\alpha}}\exp{(\nu^{1/\alpha})}\;d\nu\times \frac{1}{t+\tau}\\
=&-{\frac{1}{\lambda}\times\Big(\frac{1}{\Gamma(z)}}\Big)_{|z=\beta-\alpha}\times \frac{1}{(t+\tau)^{\alpha-\beta+1}},\quad \forall t\geq 1.
\end{align}
For $I^1_2(t)$, by using the variable $s=\frac{u^{1/\alpha}}{t}$, we have
\begin{align*}
I^1_2(t)=&-\frac{1}{\lambda^2}\frac{1}{2\pi i}\int_{\gamma(\varepsilon^\alpha,\frac{\alpha\pi}{2}+\delta)} \frac{u^{\frac{2\alpha-\beta}{\alpha}} \exp{((1+\frac{2\tau}{t})u^{1/\alpha})}}{t^\alpha}\frac{1}{\alpha t}u^{\frac{1}{\alpha}-1}\;du\\
=&-\frac{1}{\lambda^2}\frac{1}{2\alpha \pi i}
\int_{\gamma(\varepsilon^\alpha,\frac{\alpha\pi}{2}+\delta)} u^{1/\alpha} \exp{((1+\frac{2\tau}{t})u^{1/\alpha})}\;du\times
\frac{1}{t^{2\alpha-\beta+1}}.
\end{align*}
Put $\nu=(1+\frac{2\tau}{t})^\alpha u$, we obtain
\begin{align}\label{e2}
\notag I^1_2(t)=&-\frac{1}{\lambda^2}\frac{1}{2\alpha \pi i}\int_{\gamma((1+\frac{2\tau}{t})^\alpha \varepsilon^\alpha, \frac{\alpha \pi}{2}+\alpha \delta)}\nu^{1/\alpha} \exp{(\nu^{1/\alpha})}\;d\nu\times \frac{1}{(t+2\tau)^{2\alpha-\beta+1}}\\
=&-\frac{1}{\lambda^2}\times \Big(\frac{1}{\Gamma(z)}\Big)_{z=\beta-2\alpha}\times \frac{1}{(t+2\tau)^{2\alpha-\beta+1}},\quad \forall t\geq 1.
\end{align}
For $I^1_3(t)$, we have
\[
I^1_3(t)=\frac{1}{\lambda^2}\times \frac{1}{2\alpha \pi i}\int_{\gamma(\varepsilon^\alpha,\frac{\alpha\pi}{2}+\delta)}\frac{u^{\frac{2\alpha-\beta+1}{\alpha}}\exp{(1+\frac{3\tau}{t})u^{1/\alpha}}}{u\exp{(\frac{\tau u^{1/\alpha}}{t})}-\lambda t^\alpha}\;du\times \frac{1}{t^{2\alpha-\beta+1}}.
\]
Note that there exists a positive constant $C_1$ such that
\[
|s^\alpha-\lambda \exp{(-\tau s)}|\geq C_1,\quad \forall s\in \gamma(\frac{\varepsilon}{t},\frac{\pi}{2}+\delta),\; \forall t\geq 1.
\]
Thus,
\[
\left|u\exp{(\frac{\tau u^{1/\alpha}}{t})}-\lambda t^\alpha\right|\geq C_1 t^\alpha |\exp{(\frac{\tau u^{1/\alpha}}{t})}|,\quad \forall u\in \gamma(\varepsilon^\alpha,\frac{\alpha\pi}{2}+\delta),\; t\geq 1.
\]
On the other hand, there exists a constant positive $C_2$ satisfying
\begin{align*}
\int_{\gamma(\varepsilon^\alpha,\frac{\alpha\pi}{2}+\delta)} |u|^{\frac{\alpha+1}{\alpha}}|\exp{((1+\frac{2\tau }{t})u^{1/\alpha})}||du|\leq C_2,\quad \forall t\geq 1.
\end{align*}
This implies that
\begin{equation}\label{e3}
|I^1_3(t)|\le \frac{C_2}{C_1|\lambda|^2 2\alpha \pi t^{3\alpha-\beta+1}},\quad \forall t\geq 1.
\end{equation}
We now estimate the quantity $I^2(t)$. Because the domain $\gamma(R,\frac{\pi}{2}+\delta)-\gamma(\frac{\varepsilon}{t},\frac{\pi}{2}+\delta)$ is a compact set in the complex plane $\C$ and $s^\alpha-\lambda \exp{(-\tau s)}$ is analytic in this set, there is a finite number of zeros of $s^\alpha-\lambda \exp{(-\tau s)}$ in $\gamma(R,\frac{\pi}{2}+\delta)-\gamma(\frac{\varepsilon}{t},\frac{\pi}{2}+\delta)$. Let us denote these zeros by $z_1,\dots,z_k$. According to \cite[Lemma 2]{Cermak}, $z_1,\dots,z_k$ are single zeros of $s^\alpha-\lambda \exp{(-\tau s)}$. From \cite[p. 101]{Kempfle}, we have
\begin{align}\label{e4}
\notag I^2(t)=& \sum_{i=1}^k \text{Res}_{z_i}\left\{\frac{\exp{(ts)}}{s^{\beta-\alpha}(s^\alpha-\lambda \exp{(-\tau s)})}\right\}\\
=& \sum_{i=1}^k \frac{\exp{(tz_i)}}{\beta {z_i}^{\beta-1}-((\beta-\alpha)\lambda z_i^{\beta-\alpha-1}-\tau \lambda z_i^{\beta-\alpha})\exp{(-\tau z_i)}},
\end{align}
where $\text{Res}_{z_i}$ is the residue at $z_i$ of $s^\alpha-\lambda \exp{(-\tau s)}$. Hence, there is a constant $C_3>0$ such that 
\[
|I^2(t)|\leq C_3,\quad \forall t\geq 1.
\]

\noindent (i) Note that $$\Big(\frac{1}{\Gamma(z)}\Big)_{|z=0}=0.$$ For $\beta=\alpha$, from \eqref{e1}, \eqref{e2}, \eqref{e3} and \eqref{e4}, we can find a constant $C_{\alpha, \lambda}>0$ such that
\[
|E^{\lambda,\tau}_{\alpha,\alpha}(t)\leq \frac{C_{\alpha,\lambda}}{t^{\alpha+1}},\quad \forall t\geq 1.
\]

\noindent (ii) Similarly, For $\beta=1$, from \eqref{e1}, \eqref{e2}, \eqref{e3} and \eqref{e4}, we can find a constant $C_{\alpha,\lambda}>0$ such that
\[
|E^{\lambda,\tau}_{\alpha,1}(t)|\leq \frac{C_{\alpha,\lambda}}{t^{\alpha}},\quad \forall t\geq 1.
\]

\noindent (iii) First, we will prove that 
\[
\int_0^1|E^{\lambda,\tau}_{\alpha,\alpha}(s)|\;ds
\]
is bounded. Consider the following cases.

\underline{\bf{The case $\tau\geq 1$.}} We have
\begin{align*}
\int_0^1|E^{\lambda,\tau}_{\alpha,\alpha}(s)|\;ds &\leq \int_0^1\sum_{0\leq k\tau\leq s}\frac{|\lambda|^k (s-k\tau)^{\alpha k+\alpha-1}}{\Gamma(\alpha k+\alpha)}H(s-k\tau)\;ds\\
&=\int_0^1 \frac{s^{\alpha-1}}{\Gamma(\alpha)}\;ds\\
&=\frac{1}{\Gamma(\alpha+1)}.
\end{align*}
\underline{\bf{The case $0\leq \tau<1$.}} Let $k_0\in\N$ be the number satisfying $k_0\tau<1$ and $(k_0+1)\tau\geq 1$. We can partition the interval $[0,1]$ into $k_0+1$ subintervals as $[0,\tau],\dots,[k_0\tau,1]$. Then,
\begin{align*}
&\int_0^1|E^{\lambda,\tau}_{\alpha,\alpha}(s)|\;ds\leq \int_0^1\sum_{0\leq k\tau\leq s}\frac{|\lambda|^k (s-k\tau)^{\alpha k+\alpha-1}}{\Gamma(\alpha k+\alpha)}H(s-k\tau)\;ds\\
\hspace*{0.5cm}&=\sum_{i=0}^{k_0-1}\int_{i\tau}^{(i+1)\tau}\sum_{0\leq k\tau\leq s}\frac{|\lambda|^k (s-k\tau)^{\alpha k+\alpha-1}}{\Gamma(\alpha k+\alpha)}H(s-k\tau)\;ds\\
&\;\;\;\;+\int_{k_0\tau}^1 \sum_{0\leq k\tau\leq s}\frac{|\lambda|^k (s-k\tau)^{\alpha k+\alpha-1}}{\Gamma(\alpha k+\alpha)}H(s-k\tau)\;ds.
\end{align*}
Furthermore, for $0\leq i\leq k_0-1$, we see that
\begin{align*}
&\int_{i\tau}^{(i+1)\tau}\sum_{0\leq k\tau\leq s}\frac{|\lambda|^k (s-k\tau)^{\alpha k+\alpha-1}}{\Gamma(\alpha k+\alpha)}H(s-k\tau)\;ds\\
\hspace*{1.0cm}&=\sum_{k=0}^i \frac{|\lambda|^k}{\Gamma(\alpha k+\alpha)}\int_{i\tau}^{(i+1)\tau}(s-k\tau)^{\alpha k+\alpha-1}\;ds\\
\hspace*{1.0cm}&=\sum_{k=0}^i\frac{|\lambda|^{k}}{\Gamma(\alpha k+\alpha+1)}\left(((i+1-k)\tau)^{\alpha k+\alpha}-((i-k)\tau)^{\alpha k+\alpha}\right),
\end{align*}
and
\begin{align*}
&\int_{k_0 \tau}^1 \sum_{0\leq k\tau<s}\frac{|\lambda|^k (s-k\tau)^{\alpha k+\alpha-1}}{\Gamma(\alpha k+\alpha)}H(s-k\tau)\;ds\\
&=\sum_{k=0}^{k_0}\int_{k_0 \tau}^1 \frac{|\lambda|^k (\tau-k\tau)^{\alpha k+\alpha-1}}{\Gamma(\alpha k+\alpha)}\;ds\\
&=\sum_{k=0}^{k_0}\frac{|\lambda|^k}{\Gamma(\alpha k+\alpha+1)}\left((1-k\tau)^{\alpha k+\alpha}-((k_0-k)\tau)^{\alpha k+\alpha}\right),
\end{align*}
which imply that
\[
\int_0^1|E^{\lambda,\tau}_{\alpha,\alpha}(s)|\;ds
\]
is bounded. Now, for $t>1$, we use the representation $$\int_0^t |E^{\lambda,\tau}_{\alpha,\alpha}(s)|\;ds=\int_0^1 |E^{\lambda,\tau}_{\alpha,\alpha}(s)|\;ds
+\int_1^t |E^{\lambda,\tau}_{\alpha,\alpha}(s)|\;ds.$$
From (i), there exists a positive constant $\hat{C}$ such that
\begin{align*}
\int_1^t |E^{\lambda,\tau}_{\alpha,\alpha}(s)|\;ds&\leq \hat{C}\int_1^t \frac{1}{s^{\alpha+1}}\;ds\\
&\leq \frac{\hat{C}}{\alpha}. 
\end{align*}
Put $C_{\alpha,\lambda}:=\int_0^1 |E^{\lambda,\tau}_{\alpha,\alpha}(s)|\;ds+\frac{\hat{C}}{\alpha}$. Then,
\[
\sup_{t\geq 0}\int_0^t |E^{\lambda,\tau}_{\alpha,\alpha}(s)|\;ds\leq C_{\alpha,\lambda}.
\]
The proof is complete.
\end{proof}

\end{document}